\documentclass[conference]{IEEEtran}
\IEEEoverridecommandlockouts

\usepackage{graphicx}
\usepackage{amsmath}
\usepackage{amssymb}
\usepackage{amsthm}
\usepackage{array}
\usepackage{bm}
\usepackage{subcaption}
\usepackage{comment}
\usepackage{mdwmath}
\usepackage{url}
\usepackage{epsfig}
\usepackage{float}
\usepackage{color}
\usepackage[all,color]{xy}

\newtheorem{theorem}{Theorem}
\newtheorem{lemma}{Lemma}

\newtheorem{corollary}{Corollary}

\newtheoremstyle{problem}
  {}
  {}
  {}
  {0pt}
  {\bfseries}
  {:}
  { }
  {\thmname{#1}\thmnumber{ #2}\thmnote{ (#3)}}

\theoremstyle{remark}
\newtheorem{remark}{Remark}
\newcommand*\diff{\mathop{}\!\mathrm{d}}

\usepackage{color}

\newif\iftechrep \techrepfalse
\long\def\iftechreport#1#2{\iftechrep #1\else #2\fi}
\catcode`T=12 \catcode`R=12	
\def\checkiftechreport#1{
\expandafter\iistechreport#1TR. \techreptrue\fi}
\def\iistechreport#1TR#2.{\def\tmp{#2}\ifx\tmp\empty\else}
\catcode`T=11 \catcode`R=11
\def\checkTR{\checkiftechreport{\jobname}}
\checkTR

\title{Joint Frequency Regulation and Economic Dispatch Using Limited Communication}

\author{Jianan Zhang, and Eytan Modiano \\ Massachusetts Institute of Technology, Cambridge, MA, USA \thanks{This work was supported by
DTRA grants HDTRA1-13-1-0021 and HDTRA1-14-1-0058, and NSF grant CNS-1735463.}}
\begin{document}
\maketitle

\begin{abstract}
  We study the performance of a decentralized integral control scheme for joint power grid frequency regulation and economic dispatch. We show that by properly designing the controller gains, after a power flow perturbation, the control achieves near-optimal economic dispatch while recovering the nominal frequency, without requiring any communication. We quantify the gap between the controllable power generation cost under the decentralized control scheme and the optimal cost, based on the DC power flow model. Moreover, we study the tradeoff between the cost and the convergence time, by adjusting parameters of the control scheme. 

  Communication between generators reduces the convergence time. We identify key communication links whose failures have more significant impacts on the performance of a distributed power grid control scheme that requires information exchange between neighbors.
\end{abstract}

\section{Introduction}
The integrations of renewable energy resources increase the fluctuations of power supply. To balance the power supply and demand, power generations are controlled using primary, secondary, and tertiary controls under different time scales. The primary control at a power generator, droop control, responds to power flow perturbations within milliseconds to seconds, and re-balances the power supply and demand at the cost of creating frequency deviation. The secondary control, Automatic Generation Control (AGC), adjusts generator setpoints to recover the nominal frequency in seconds to minutes. The tertiary control, economic dispatch, minimizes the total power generation cost by scheduling an operating point for each generator, and operates in minutes to an hour.

Communication is essential for frequency regulation and economic dispatch. Both the AGC and the economic dispatch are traditionally implemented using centralized control. The control center gathers information from all generators and loads, and computes setpoints for generators to adjust to disturbances. Both the information aggregation and setpoints delivery require communication between the control center and controllable nodes.

There have been recent advancement in developing distributed and decentralized frequency control techniques \cite{bejestani2014hierarchical, dorfler2016breaking, andreasson2013distributed}. Motivated by the need to adapt to more frequent power fluctuations and faster response, some of these controllers require communication between neighbor nodes, to achieve frequency control for power grids with renewable integrations.

There are two major categories of distributed and decentralized frequency control -- \emph{primal-dual controller} and \emph{integral controller}. By formulating the frequency control as a convex optimization problem, a primal-dual algorithm was developed in \cite{zhao2014design} for joint frequency regulation and economic dispatch. The primal-dual controller was extended to handle power transmission line thermal limits and inter-area flow constraints in \cite{mallada2017optimal}. By considering frequency regulation and economic dispatch in different time scales, a primal-dual controller under stochastic power demand was developed in \cite{cai2017distributed}. For these primal-dual controllers, communication between adjacent nodes is required to transmit Lagrangian multipliers. Communication between a group of nodes (not necessarily adjacent nodes) is needed to handle more complicated constraints (e.g., inter-area flows).

Integral controller utilizes local frequency deviation information to adjust the controllable power generation or load \cite{dorfler2016breaking,andreasson2014distributed,zhao2015distributed}. In general, a decentralized integral controller is able to recover the nominal frequency based on local measurement, but unable to achieve the optimal operating point where the cost is minimized. By communicating marginal generation costs between nearby controllable nodes, the distributed averaging-based integral control achieves both the frequency regulation and economic dispatch, if all the controllable nodes are connected by a communication network~\cite{zhao2015distributed}.

Economic dispatch or power sharing can be be achieved by a decentralized droop control, under specific droop coefficients \cite{simpson2013synchronization,dorfler2016breaking}. The frequency deviation serves as a common reference for the power sharing among all generators. The work closest to ours is the study of a decentralized leaky integral control in \cite{ainsworth2013design}. The leaky integral control can achieve both power sharing and arbitrarily small frequency deviation in the steady state. In contrast, we study an integral control that recovers the nominal frequency.

Although either frequency regulation or economic dispatch can be achieved by decentralized control \cite{dorfler2016breaking, zhao2015distributed}, communication is needed to achieve both objectives.
Loss of communication may lead to sub-optimal control. Using power line measurement, a control policy was developed in \cite{parandehgheibi2016distributed} to withstand any single communication link failure. The role of communication network topology on power grid control has been studied in \cite{lian2017game, lin2013design}.

In this paper, we study the performance of a decentralized integral controller with properly designed controller gains, for minimizing the adjustable power generation cost in the steady state. We quantify the gap between the cost under the decentralized control and the minimum possible cost, and derive conditions for joint frequency regulation and economic dispatch, based on the DC power flow model. We study the tradeoff between the cost and the convergence time, by changing the parameters of the controller. We also study the effectiveness of communication on reducing the convergence time, and quantify the importance of each individual communication link in a distributed control that require information exchange between neighbors. The method can be generalized to handle arbitrary convex power generation costs and power generation capacity constraints. Moreover, we observe that a delayed integral control scheme achieves near-optimal generation cost using significantly smaller convergence time. 

The rest of the paper is organized as follows. In Section \ref{sc:model}, we describe the system model. In Section \ref{sc:control}, we describe a decentralized integral control scheme and study its performance. In Section \ref{sc:comm}, we study a distributed integral control scheme aided by communication between nodes, and characterize the importance of each individual communication link. In Section \ref{sc:extension}, we extend the integral control to handle arbitrary convex costs and generation capacity constraints, and develop a delayed control policy. Section \ref{sc:simu} presents simulation results. Section \ref{sc:conclusion} concludes the paper.

\section{Model} \label{sc:model}
The power grid is modeled by a connected graph $G(V,E)$, which has $n = |V|$ nodes and $m = |E|$ edges. Each node represents a bus, which is connected to a generator or a load. Each edge represents a power line. Let $V_G \subset V$ denote the generators and $V_L \subset V$ denote the loads. We assume that lines are lossless and denote the absolute value of the susceptance of power line $(j,k)$ by $B_{jk}$. We consider an arbitrary orientation of power lines. A positive power flow on a power line indicates a flow in the same orientation as the power line, and a negative power flow indicates a flow in the opposite orientation. Bus voltages are normalized to 1 pu (per unit).

Let $\omega_j$ denote the frequency deviation from the nominal frequency at bus $j$. Let $\theta_j$ denote the phase angle with respect to the rotating framework of nominal frequency (i.e., $\theta_j(t) = (\theta_j + 2 \pi \cdot 60\text{Hz} \cdot t) \mod 2 \pi$). Let $p_j$ denote the unadjustable power generation or load, and let $u_j$ denote the controllable power generation or load, which take a positive value for net generation and a negative value for net load. Before disturbance, $u_j = 0, \forall j \in V$. We consider a DC power flow model. The power dynamics at a generator, which has moment of inertia $M_j$ and droop coefficient $D_j$, follow the swing equation
\begin{equation}\label{eq:swing}
  M_j \dot{\omega}_j = - D_j \omega_j + p_j + u_j - \sum_{k \in V} B_{jk} (\theta_j - \theta_k),~~\forall j \in V_G.
\end{equation}
The power dynamics at a load, which has a linear frequency-dependent load coefficient $D_j$, follow the equation
\begin{equation}\label{eq:load}
  0 = - D_j \omega_j + p_j + u_j - \sum_{k \in V} B_{jk} (\theta_j - \theta_k),~~\forall j \in V_L.
\end{equation}

We study the frequency regulation and economic dispatch problems after a power flow perturbation. The objective of frequency regulation is to recover the nominal frequency at all locations. The objective of economic dispatch is to minimize the total cost of adjustable generation and load. For simplicity, we consider the minimization of the sum of quadratic cost functions, where $a_j$ is the cost coefficient at $j$.
\begin{eqnarray}
  \min_{u,\theta} & \sum_{j \in V} \frac{1}{2} a_j u^2_j \label{eq:dispatch}\\
  \text{s.t.} & p_j + u_j - \sum_{k \in V}B_{jk} (\theta_j - \theta_k) = 0,~~ j \in V. \label{eq:balance}
\end{eqnarray}
The power balance constraints Eq.~(\ref{eq:balance}) guarantee frequency recovery. This can be verified by noticing $\omega = 0$ and $\dot{\omega} = 0$ in Eqs. (\ref{eq:swing}) and (\ref{eq:load}) if Eq.~(\ref{eq:balance}) holds in the steady state.

The \emph{marginal cost} of power generation is the rate of change in cost by increasing the net generation. In the optimal solution, the marginal costs of power generation are identical at all locations ($ \diff (a_j u_j^2 / 2) / \diff u_j = a_j u_j = a_k u_k, \forall j,k \in V$). We ignore power line thermal limits and generator capacity constraints for simplicity. In Section \ref{sc:extension}, we generalize the methods to minimize arbitrary convex costs, and consider generator capacity constraints.

\section{Decentralized integral control} \label{sc:control}
Throughout this paper, we study the control after a perturbation of power generation or load. We assume that the initial power flows are balanced ($\sum_{j\in V}p^0_j = 0$). After a perturbation of generation or load, by controlling the adjustable power $u$, $\sum_{j\in V} (p_j + u_j) = 0$ holds in the steady state.
We aim to develop a control policy that achieves both frequency regulation and economic dispatch, by properly setting the adjustable power while adhering to the power flow dynamics Eqs.~(\ref{eq:swing}) and (\ref{eq:load}).

A decentralized frequency integral controller Eq.~(\ref{eq:control}) was studied in \cite{zhao2015distributed}. The controller measures the local frequency deviation $\omega$, and adjusts $u$ according to the measurement. It has been shown in \cite{zhao2015distributed} that the controller converges to the steady-state and recovers the nominal frequency for any $K_j > 0$, due to the negative feedback loop. We show in this paper that by properly setting $K_j$, the controller achieves near-optimal economic dispatch.
\begin{eqnarray}
  \dot{u}_j &=& - K_j \omega_j, ~~\forall j \in V. \label{eq:control}
\end{eqnarray}
We prove the following theorem.
\begin{theorem}\label{th:dec}
  For $K_j = h/a_j$, the steady-state cost under the decentralized control is at most $4(\Delta p)^2 n h / (b \lambda_2)$ more than the optimal cost, where $\Delta p$ is the initial power change at any bus, ${\lambda}_2$ is the algebraic connectivity of the unweighted graph $G$, $h > 0$, and $b$ is the minimum absolute value of the power line susceptance.
\end{theorem}


\begin{remark}
  The decentralized control achieves frequency regulation and near-optimal economic dispatch, if $K_j = h/a_j$, $h > 0$, and either of the two conditions are satisfied:

  1) the absolute values of power line susceptances are large.

  2) $h$ is small.
\end{remark}

\begin{remark}
By setting $h$ small, the gap $4(\Delta p)^2 n h / (b \lambda_2)$ becomes small. However, the convergence time increases, because the controller gain in Eq.~(\ref{eq:control}) is small. There is a tradeoff between the cost and the convergence time. In Section \ref{sc:comm}, we study the effects of communication in reducing the convergence time.
\end{remark}

\begin{remark}
Previous work \cite{dorfler2016breaking} studied a method for economic dispatch using decentralized droop control, by properly setting the droop coefficients. There exists a non-zero frequency deviation in the steady state, and the common frequency deviation at all buses serves as a reference for power sharing or cost minimization. Our methods are significantly different from \cite{dorfler2016breaking}. Instead of using global consensus information (i.e., frequency deviation), we study the properties of power flows in steady states, and utilize the invariance Eq.~(\ref{eq:diffAngle}) to design the controller gains to minimize the cost.
\end{remark}

In the rest of the section, we first present the intuition and preliminaries for the performance analysis of the decentralized controller, and then provide the proof of the theorem.

\subsection{Preliminary and intuition}
Let $C$ be the network incidence matrix, which has $n$ rows and $m$ columns. Suppose that the $l$-th edge is oriented from node $j$ to node $k$. Then $C_{jl} = 1$ and $C_{kl} = -1$. Let $B$ be an $m \times m$ diagonal matrix, whose $l$-th diagonal represents the absolute value of the $l$-th power line susceptance. Let $\theta^0$ denote the initial phase angles before the perturbation, and let $\theta$ denote the phase angles in the steady state after the perturbation. In the steady states, the frequency stays fixed at the nominal frequency ($\omega = \dot{\omega} = 0$), and the power flows are balanced at each bus.
$$ p^0 = C B C^\top \theta^0; ~~~~~~~~p + u = C B C^\top \theta. $$
Subtracting the two equations,
\begin{equation}\label{eq:diff}
CBC^\top(\theta - \theta^0) = p + u - p^0.
\end{equation}
\iftechreport{}{The phase angle difference is given by
\begin{equation}\label{eq:diffAngle}
  \theta - \theta^0 = (CBC^\top)^+ (p - p^0 + u) + c 1_{n \times 1},
\end{equation}
proved in the technical report \cite{report}.}

\iftechreport{The phase angle difference is given by Lemma \ref{th:diff}.
\begin{lemma}\label{th:diff}
The difference of phase angles $\theta - \theta^0$ can be determined up to a constant shift. I.e.,
\begin{equation}\label{eq:diffAngle}
  \theta - \theta^0 = (CBC^\top)^+ (p - p^0 + u) + c 1_{n \times 1},
\end{equation}
where $(CBC^\top)^+$ denotes the pseudo-inverse of $CBC^\top$.
\end{lemma}

\begin{proof}
  For a connected graph $G(V,E)$, $C$ has rank $n - 1$. Since $B$ is diagonal and positive definite, the graph Laplacian matrix $C B C^\top$ has rank $n - 1$. The nullspace of $C B C^\top$ has dimension 1 and is spanned by the vector $1_{n \times 1}$. To prove that $\theta - \theta^0$ given by Eq. (\ref{eq:diffAngle}) is the solution to Eq. (\ref{eq:diff}), it suffices to verify that
  \begin{equation}\label{eq:pinv2}
    (CBC^\top)(CBC^\top)^+ (p - p^0 + u) = p - p^0 + u.
  \end{equation}

  Using linear algebra techniques,
  \begin{equation}\label{eq:pinv}
  (CBC^\top)(CBC^\top)^+ = I - \frac{1}{n} J,
  \end{equation}
  where $I$ is an $n \times n$ identity matrix, and $J$ is an $n \times n$ matrix with all one elements. See Lemma 3 in \cite{gutman2004generalized} for a proof for unweighted graph Laplacian. The same techniques can be used to prove the weighted graph Laplacian in Eq. (\ref{eq:pinv}).

  Since the power flows are balanced in the steady states, $\sum_{j \in V} p^0_j = \sum_{j\in V}(p_j + u_j) = 0$. Therefore, $J(p - p^0 + u) = 0$. Since $I(p - p^0 + u) = p - p^0 + u$, we have proved Eq. (\ref{eq:pinv2}).
\end{proof}
}{}

Let $K$ be an $n \times n$ diagonal matrix whose $j$-th diagonal equals $K_j$. Given the control policy Eq. (\ref{eq:control}), in the steady state, the amount of adjustable power is given by
\begin{equation}\label{eq:controlPower}
  u = - K(\theta - \theta^0).
\end{equation}

If $K_j = h/a_j$, then $a_j u_j = h({\theta}^0_j - {\theta}_j)$. If the diagonals of $B$ are large, $(CBC^\top)^+ (p - p^0 + u)$ is small and ${\theta} - {\theta}^0$ is almost equal to $c 1_{n \times 1}$. The marginal costs at all generators $a_j u_j$ are almost the same, thus achieving near-optimal economic dispatch.

\subsection{Proof of Theorem \ref{th:dec}}
  We consider a power perturbation at node $k$ and denote the amount of power change by $\Delta p$. Without loss of generality, we assume that $\Delta p < 0$ (i.e., load increase or generation decrease). After the change, the frequency drops below the nominal frequency and $u > 0$. In the steady state after the change, $\sum_{j \in V} u_j = - \Delta p$. The $L_1$ norm of the vector $p - p^0 + u$ is at most $\sum_{j \in V} |p_j - p_j^0 + u_j| = \sum_{j \in V, j \neq k} u_j + |\Delta p + u_k| \leq 2|\Delta p|$. Suppose that the absolute value of every element of $(CBC^\top)^+$ is at most $M$. Then, the absolute value of every element in $(CBC^\top)^+ (p - p^0 + u)$ is at most $2M|\Delta p|$. Therefore,
  $$ |(\theta_i - \theta_i^0) - (\theta_j - \theta_j^0)| \leq 4M|\Delta p|, ~~ \forall i,j \in V. $$

  The difference of the marginal costs at $i$ and $j$ is at most
  \begin{eqnarray*}
    |a_i u_i - a_j u_j| &=& h |(\theta_i - \theta_i^0) - (\theta_j - \theta_j^0)|\\
     &\leq& 4hM |\Delta p|
  \end{eqnarray*}

  Since the marginal costs differ by at most $4hM |\Delta p|$, and the marginal cost at each node is an increasing function in the generation amount, the cost saving in dispatching one unit power generation to a different node is at most $4hM |\Delta p|$. The total amount of adjustable generation is $\sum_j u_j = |\Delta p|$. Therefore, the generation cost under the control Eq.~(\ref{eq:control}) is at most $4hM(\Delta p)^2$ higher than the optimal generation cost.

  Next we bound $M$. By spectral decomposition, the symmetric matrix $CBC^\top = UDU^\top$, where $U$ is an orthonormal matrix and $D$ is a diagonal matrix. The diagonals of $D$ are the eigenvalues of $CBC^\top$. Let $\lambda'$ denote the set of eigenvalues. Let $v_i$ be the $i$-th column of $D$.
  $$ CBC^\top = \sum_{i = 1}^{n} \lambda'_i v_i v_i^\top.$$

  Moreover, the graph Laplacian matrix $CBC^\top$ is positive semi-definite and has rank $n-1$, which has the smallest eigenvalue $\lambda'_1 = 0$ and $n-1$ positive eigenvalues. The pseudo-inverse of $CBC^\top$ is given by
  $$ (CBC^\top)^+ = \sum_{i = 2}^{n} (1/\lambda'_i) v_i v_i^\top.$$

  Let $L = C C^\top$ be the Laplacian of the unweighted graph $G$. The second smallest eigenvalue is the algebraic connectivity of $G$, and is given by
  \begin{equation} \label{eq:algb}
  {\lambda}_2 = \min\{\frac{y^\top L y }{y^\top y} | y \neq 0, 1_{1 \times n} y = 0\}.
  \end{equation}

  Let $b$ be the smallest absolute value of susceptance. \iftechreport{$CBC^\top = bL + L'$. The matrix $L'$ can be viewed as the Laplacian of the weighted graph where edge $(j,k)$ has a weight $B_{jk} - b \geq 0$, and is positive semi-definite. The second smallest eigenvalue of $CBC^\top$ is bounded by Eq.~(\ref{eq:eig4}). The vector $y^*$ in Eq.~(\ref{eq:eig2}) is the vector that achieves the minimum in Eq.~(\ref{eq:eig1}). Inequality~(\ref{eq:eig3}) follows from that $L'$ is positive semi-definite. Inequality (\ref{eq:eig4}) follows from Eq. (\ref{eq:algb}).
  \begin{eqnarray}
    \lambda'_2 &=& \min \{\frac{y^\top CBC^\top y }{y^\top y} | y \neq 0, 1_{1 \times n} y = 0\} \label{eq:eig1}\\
    &=& \frac{y^{*\top} (bL + L') y^* }{y^{*\top} y^*} \label{eq:eig2}\\
    &\geq& \frac{y^{*\top} bL y^* }{y^{*\top} y^*} \label{eq:eig3}\\
    &\geq& b {\lambda}_2. \label{eq:eig4}
  \end{eqnarray}

}{We show in \cite{report} that $\lambda'_2 \geq b \lambda_2$.}
Since $||v_i||_2 = 1$, the absolute value of every element in $v_i v_i^\top$ is at most 1. The absolute value of every element in $(CBC^\top)^+$ is at most $M \leq \sum_{i=2}^n (1/\lambda'_i) \leq n/(b {\lambda}_2)$. Therefore, the generation cost under the controller is at most $4(\Delta p)^2 n h / (b \lambda_2)$ higher than the optimal cost. 

\section{Distributed control under partial communication}
\label{sc:comm}
In the previous section, we studied an integral controller that adjusts controllable power based on local measurement and does not require any communication. In this section, we study the benefit of communication in power grid control. Communication is useful to exchange the marginal cost information between controllable nodes. It reduces the convergence time of the control, by eliminating the need to use a small controller gain for economic dispatch.

Consider a communication network $G'(V',E')$, where $V'=V$ denotes the buses in the power grid, and $E'$ denotes the communication links. It has been shown that by exchanging the marginal costs between neighbors, a distributed averaging-based integral control can achieve both frequency regulation and economic dispatch, if $G'$ is connected \cite{zhao2015distributed}. In this section, we develop a control policy under the \emph{failures} of communication links, where the remaining communication links do not connect all the nodes. 

Let $E^*$ denote the minimum set of links that are parallel to power lines and merge the disjoint communication components into a connected graph. Let $V^*$ denote the nodes adjacent to $E^*$. Notice that $E^*$ and $V^*$ are non-empty if and only if $G'(V',E')$ is disconnected. See Fig. \ref{fig:comp} for an illustration.

\begin{figure}[h]
\centering
\includegraphics[width=.8\linewidth]{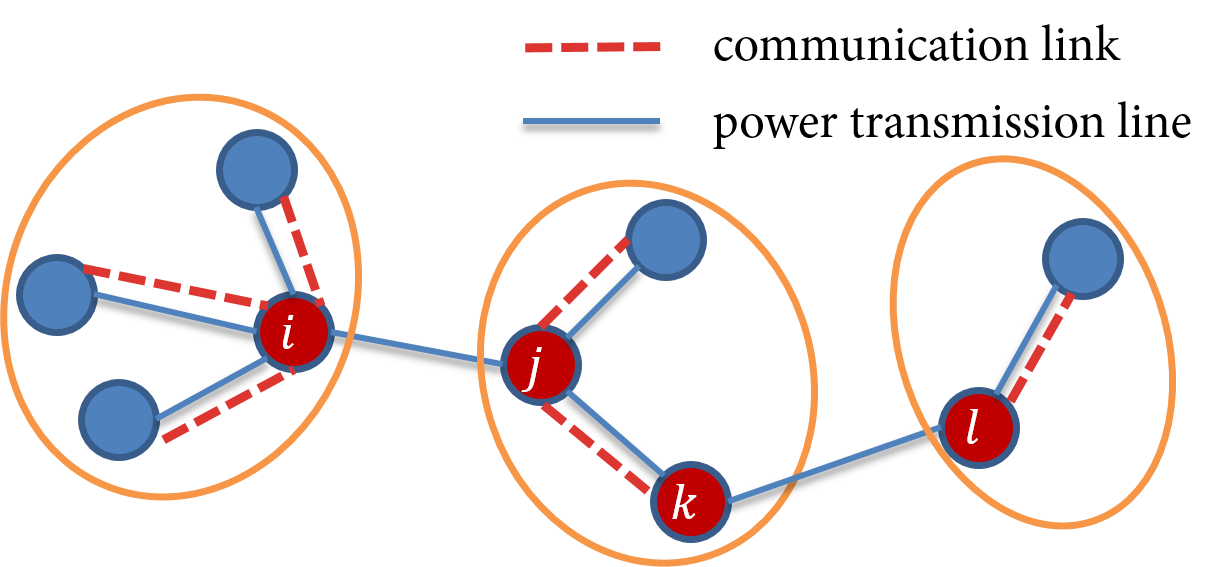}
\caption{Illustration of communication components. Links $E^* = \{(i,j), (k,l)\}$ connect three disjoint communication components. Nodes $V^* = \{i,j,k,l\}$ are their adjacent nodes.}
\label{fig:comp}
\end{figure}

We study the performance of the control policy given by Eqs. (\ref{eq:newcontrol1}) and (\ref{eq:newcontrol2}), where $K_j = h/a_j$. The integral control in Section \ref{sc:control} is applied to nodes $V^*$ (i.e., Eq.~\ref{eq:newcontrol1}). However, for nodes that are connected by communication links, i.e., $V \setminus V^*$, the distributed averaging controller of \cite{zhao2015distributed} is used, and nearby nodes exchange the marginal costs $a_j u_j$. In the steady state, $a_j u_j = a_k u_k$ for $j$ and $k$ in the \emph{same} communication component, by the analysis in \cite{zhao2015distributed}. The key is to bound the gap between the marginal costs in different components.

\begin{eqnarray}
  \dot{u}_j &=& - K_j \omega_j, ~~ \forall j \in V^*, \label{eq:newcontrol1}\\
  \dot{u}_j &=& - K_j \omega_j - \hspace{-3mm} \sum_{(j,k) \in E}\hspace{-2mm}(a_j u_j - a_k u_k), \forall j \in V \setminus V^*. \label{eq:newcontrol2}
\end{eqnarray}


For simplicity, we assume that the communication network initially has the same topology as the power grid. Suppose that the communication link between $i$ and $j$ fails, and $i$ and $j$ are separated in two communication components. Then, $i,j \in V^*$. By the analysis in Section \ref{sc:control}, $a_i u_i$ and $a_j u_j$ are given by $h(\theta_i^0 - \theta_i)$ and $h(\theta_j^0 - \theta_j)$, respectively.

By left-multiplying both sides of Eq. (\ref{eq:diffAngle}) by $BC^\top$, Eq. (\ref{eq:inv}) holds for any control policy on $u$. Let the $i$-th diagonal of the diagonal matrix $D$ be $D_{i} = \sqrt{B_{i}}$, which satisfies $B = DD^\top$.
\begin{eqnarray}\label{eq:inv}
BC^\top(\theta - \theta^0) &=& BC^\top (CBC^\top)^+ (p + u - p^0).\\
BC^\top(\theta - \theta^0) &=& D (CD)^+ (p + u - p^0).\nonumber
\end{eqnarray}
Let $y$ denote the $m \times 1$ vector $D (CD)^+ (p + u - p^0)$. Recall that $C$ is an adjacency matrix with $C_{il} = 1$ and $C_{jl} = -1$ if the $l$-th edge is oriented from $i$ to $j$. The susceptance of the power line that connects $i$ and $j$ is the $l$-th diagonal value $B_{l}$. We obtain
$$(\theta_i - \theta_i^0) - (\theta_j - \theta_j^0) = y_l / B_{l}. $$

If $B_{l}$ is large, then $|(\theta_i - \theta_i^0) - (\theta_j - \theta_j^0)|$ is small. Recall Eq. (\ref{eq:controlPower}). Under the integral control Eq. (\ref{eq:newcontrol1}), the gap between the marginal costs at $i$ and $j$ (i.e., $|a_i u_i - a_j u_j|$) is small. Intuitively, given a bounded power flow on the power line $B_l(\theta_i - \theta_j)$, if the susceptance $B_l$ is large, the difference between phase angles $\theta_i - \theta_j$ is small. Therefore, the difference of phase angles $\theta_i - \theta_i^0$ is close to $\theta_j - \theta_j^0$, which indicates a small gap in the marginal costs at nodes $i$ and $j$.

To conclude, under the control policy given by Eqs. (\ref{eq:newcontrol1}) and (\ref{eq:newcontrol2}), the failure of a communication link has less severe impacts if its associated power line has a large susceptance.

\section{Variations of the integral control}\label{sc:extension}
In this section, we extend the the decentralized control in Section \ref{sc:control} to handle arbitrary convex costs for adjustable power and generator capacity constraints. Moreover, we study the benefit of delayed control on minimizing the total costs.
\subsection{Arbitrary convex costs and generator capacity constraints}
Let $f_j(u)$ denote the cost of increasing the generation (or decreasing the load) by $u$ at node $j$. We assume that $f_j(u)$ is strictly convex and differentiable, and attains the minimum at $f_j(0) = 0$, $\forall j \in V$. The derivative $g_j(u) = f'_j(u)$ is monotonically increasing, and the inverse $g_j^{-1}(v)$ is well defined.

We study a control policy given by Eqs. (\ref{eq:virtualcontrol2}) and (\ref{eq:control2}).
\begin{eqnarray}
  \dot{v}_j &=& - h \omega_j, ~~~~\forall j \in V. \label{eq:virtualcontrol2}\\
  u_j &=& g_j^{-1}(v_j), ~~\forall j \in V. \label{eq:control2}
\end{eqnarray}
The controller gain $h$ is positive and identical at all nodes. For the special case of quadratic cost $f_j(u_j) = a_j u_j^2 / 2$, $g_j(u_j) = a_j u_j$, the control is equivalent to Eq.~(\ref{eq:control}) with $K_j = h/a_j$.

The controller measures the local frequency deviation $\omega_j$, and then adjusts a virtual price $v_j$. The virtual price serves as a reference for the controllable power generation. The marginal cost of power generation at node $j$ is $v_j$, guaranteed by Eq.~(\ref{eq:control2}).

If the power line susceptances are large, or the controller gain $h$ is small, the virtual prices $v$ and the marginal costs at different controllable nodes are close. Thus, the total cost is approximately minimized. \iftechreport{More precisely,
\begin{corollary}\label{th:dec2}
  For a strictly convex and differentiable cost $f_j(u)$ that attains the minimum at $f_j(0)=0$, $\forall j \in V$, the steady-state cost under the decentralized control (\ref{eq:virtualcontrol2}) and (\ref{eq:control2}) is at most $4(\Delta p)^2 n h / (b \lambda_2)$ more than the optimal cost, where $\Delta p$ is the initial power change at any bus, ${\lambda}_2$ is the algebraic connectivity of unweighted graph $G$, $h > 0$, and $b$ is the minimum absolute value of the power line susceptance.
\end{corollary}
\begin{proof}
  The function $g_j(u_j) = f'_j(u_j)$ denotes the rate of cost change at node $j$ as the amount of net adjustable generation increases. 
  We aim to prove that the difference of the marginal costs at $i$ and $j$ is
  \begin{eqnarray}\label{eq:marginal}
    |f_i'(u_i) - f_j'(u_j)| &=& h |(\theta_i - \theta_i^0) - (\theta_j - \theta_j^0)|,
  \end{eqnarray}
  where $\theta^0$ are the phase angles before perturbation and $\theta$ are the phase angles in the steady state after the perturbation. The rest of the proof follows from the proof of Theorem \ref{th:dec}.

  Under Eq.~(\ref{eq:virtualcontrol2}), $v_j = -h(\theta_j - \theta_j^0)$. We obtain
  \begin{eqnarray*}
    |v_i - v_j| &=& h |(\theta_i - \theta_i^0) - (\theta_j - \theta_j^0)|.
  \end{eqnarray*}
  Since $f_j(u)$ is strictly convex and differentiable, $g_j(u) = f_j'(u)$ is monotonically increasing, and is a bijection function. 
  According to Eq.~(\ref{eq:control2}),
  $$ f_j'(u_j) = g_j(u_j) = v_j, ~~\forall j \in V.$$
  Therefore, we have proved Eq.~(\ref{eq:marginal}).
\end{proof}
}{Theorem \ref{th:dec} holds under the new control. See \cite{report} for proof details.}

To handle the power generation capacity constraints, it suffices to replace Eq.~(\ref{eq:control2}) by the following equation.
$$u_j = \max(c^1_j, \min(c^2_j, g_j^{-1}(v_j)), ~~\forall j \in V, $$
where $[c^1_j,c^2_j]$ is range of controllable net generation at node $j$. Under the same analysis, the total cost is approximately minimized, and the frequency is recovered to the nominal frequency in the steady state, as long as it is feasible to balance the power generation and load under the capacity constraints.

\subsection{Delayed control}
The integral of frequency deviation is utilized at each controllable node to serve as a reference for the marginal cost of adjustable power. In previous sections, we studied conditions for the references to be nearly identical at all locations in order for economic dispatch. Next, we study a controller that only adjusts the controllable power using frequency deviation after a timeout period $T$, given by Eqs. (\ref{eq:control3}) and (\ref{eq:control4}).
\begin{eqnarray}
\dot{u}_j(t) &=& 0, ~~~~~~~~~~~~~~~~t \leq T, \forall j \in V. \label{eq:control3}\\
  \dot{u}_j(t) &=& - K_j \omega_j, ~~~~~~~~~t > T, \forall j \in V. \label{eq:control4}
\end{eqnarray}

The intuition is that, after some time $T$ without any control on $u$, the frequency deviations at all nodes become almost identical. The deviations could serve as references to adjust $u$. In the numerical result section, we observe significant cost savings by the delayed control, at similar convergence time compared with the original integral control.


\section{Numerical results} \label{sc:simu}
In this section, we verify the performance of the controllers using a simple example with 10 nodes and 10 edges. 
The network topology (Fig. \ref{fig:topology}) and the data are identical to those in \cite{parandehgheibi2016distributed}. We study the control after a perturbation of 5 units load increase at node 3. The minimum sum of quadratic costs shown in Eq. (\ref{eq:dispatch}) is 23.27. \iftechreport{The data are presented below.

For 10 nodes (white numbers indicate node ID), \\
Inertia $M = \{0.01,0.02,0.01,0.1,0.05,0.8,0.05,1,0.1,0.01\}$.\\
Initial power $p = \{1,5,−2,6,−5,−10,−4,8,5,−4\}$.\\
Droop coefficient $D = \{0.33,1.67,0.67,2.00,1.67,\\
~~~~~~~~~~~~~~~~~~~~~~~~~~~~~~3.33,1.33,2.67,1.67,1.33\}$.\\
Cost coefficient $a = \{20,20,200,200,10,20,14,18,10,20\}$.

For 10 power lines (black numbers indicate line ID),\\
Absolute values of susceptance $B = \{1.00,0.50,0.33,\\
1.00,0.20,0.25,0.17,1.00,0.11,1.00\}$.
}{The data are also presented in \cite{report}.}

\begin{figure}[h]
\centering
\includegraphics[width=.5\linewidth]{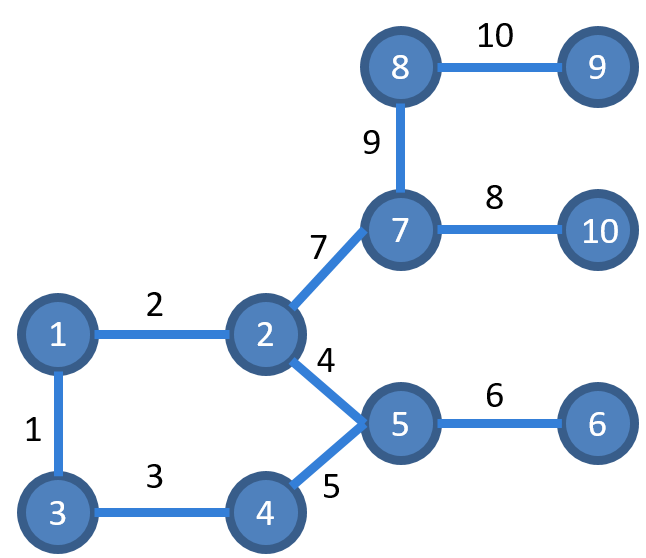}
\caption{Topology of the power grid.}
\label{fig:topology}
\end{figure}

\subsection{Cost vs. controller gain and power line susceptance}
We evaluate the total costs in the steady states after the perturbation, for different values of controller gains. In Fig. \ref{fig:gain}, we observe that as $h$ decreases, the cost under the integral control Eq. (\ref{eq:control}) approaches the optimal cost. The near-linear dependence on $h$ matches the prediction in Theorem \ref{th:dec}. \iftechreport{}{We observe a similar curve by changing the susceptance, which is further discussed in \cite{report}.}
\begin{figure}[h]
\centering
\includegraphics[width=.65\linewidth]{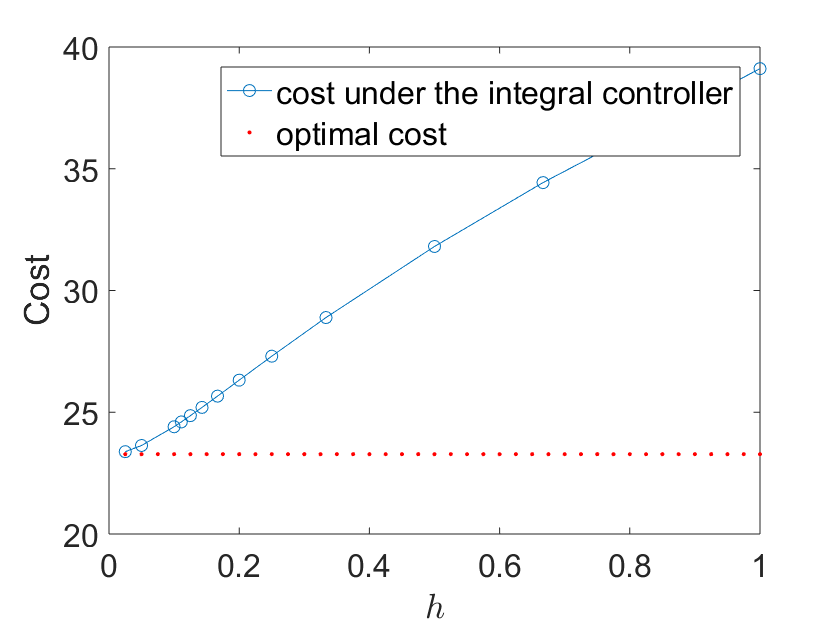}
\caption{Cost decreases as the controller gain decreases.}
\label{fig:gain}
\end{figure}

\iftechreport{
By dividing all line susceptances by the values in the $x$-axis, the cost decreases and follows the same curve as Fig. \ref{fig:gain}. This can be explained analytically. From Eq. (\ref{eq:controlPower}), we obtain
$$ \theta - \theta^0 = -K^{-1}u.$$
Moreover,
$$ p + u - p^0 = (CBC^\top)(\theta - \theta^0).$$
Therefore,
$$ (I + CBC^\top K^{-1})u = p^0 - p,$$
where $I$ is the identity matrix.
Since $C(\alpha B)C^\top K^{-1} = CBC^\top (K/\alpha)^{-1} = \alpha CBC^\top K^{-1}$, the adjustable power $u$ are identical under 1) power line susceptance $\alpha B$ and controller gain $h$; 2) power line susceptance $B$ and controller gain $h/\alpha$, for any positive scaler $\alpha$.
}{}
\subsection{Reducing convergence time using communication}
We study the role of communication in reducing the convergence time to reach the steady state while guaranteeing a low cost. We consider a connected communication network that has the same topology as the power grid. Figure \ref{fig:conv} illustrates the change of adjustable power at all nodes as time increases, under the control Eqs. (\ref{eq:newcontrol1}) and (\ref{eq:newcontrol2}). The four figures correspond to the scenario where there is no communication link failure, links $\{2, 4\}$ failure, links $\{2,4,9 \}$ failure, and all links failure, respectively.

The cost under the control with a connected communication network (Fig. \ref{fig:1}) is 23.27, with convergence time around 200 seconds under $h = 1$. The costs for the other scenarios under communication link failures are set to be around 24.43, which is $5\%$ higher than the optimal cost. In order to achieve the target cost, $h$ is set to be $1/1.7, 1/8.5, 1/9.8$, respectively, and the convergence times are around 250, 600, and 750 seconds, respectively, for Figs. \ref{fig:2}, \ref{fig:3}, and \ref{fig:4}. We observe that the convergence time increases as there are more communication link failures, in order to guarantee the same target cost.  

\begin{figure}[h]
  \begin{subfigure}[b]{0.24\textwidth}
    \includegraphics[width=\textwidth]{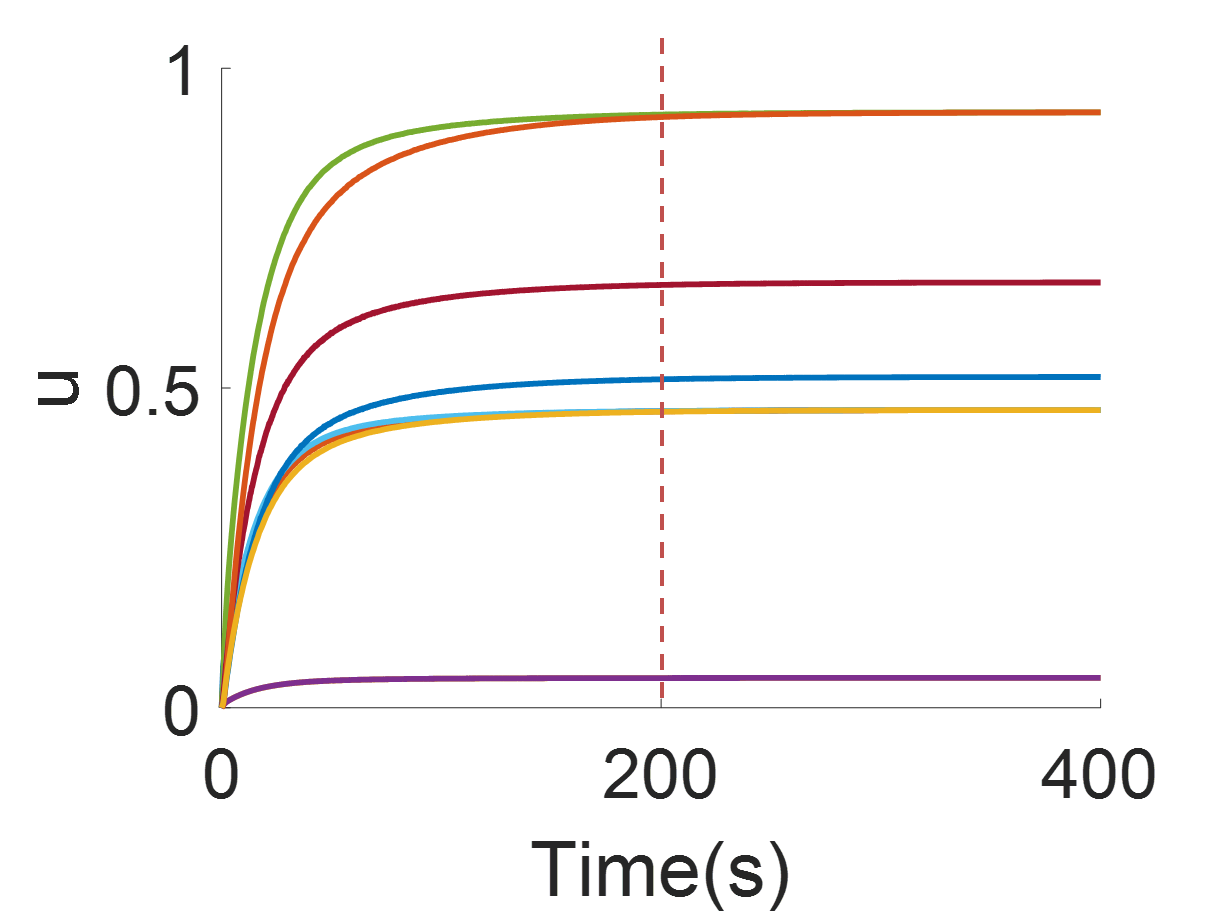}
    \caption{Full communication.}
    \label{fig:1}
  \end{subfigure}
  \begin{subfigure}[b]{0.24\textwidth}
    \includegraphics[width=\textwidth]{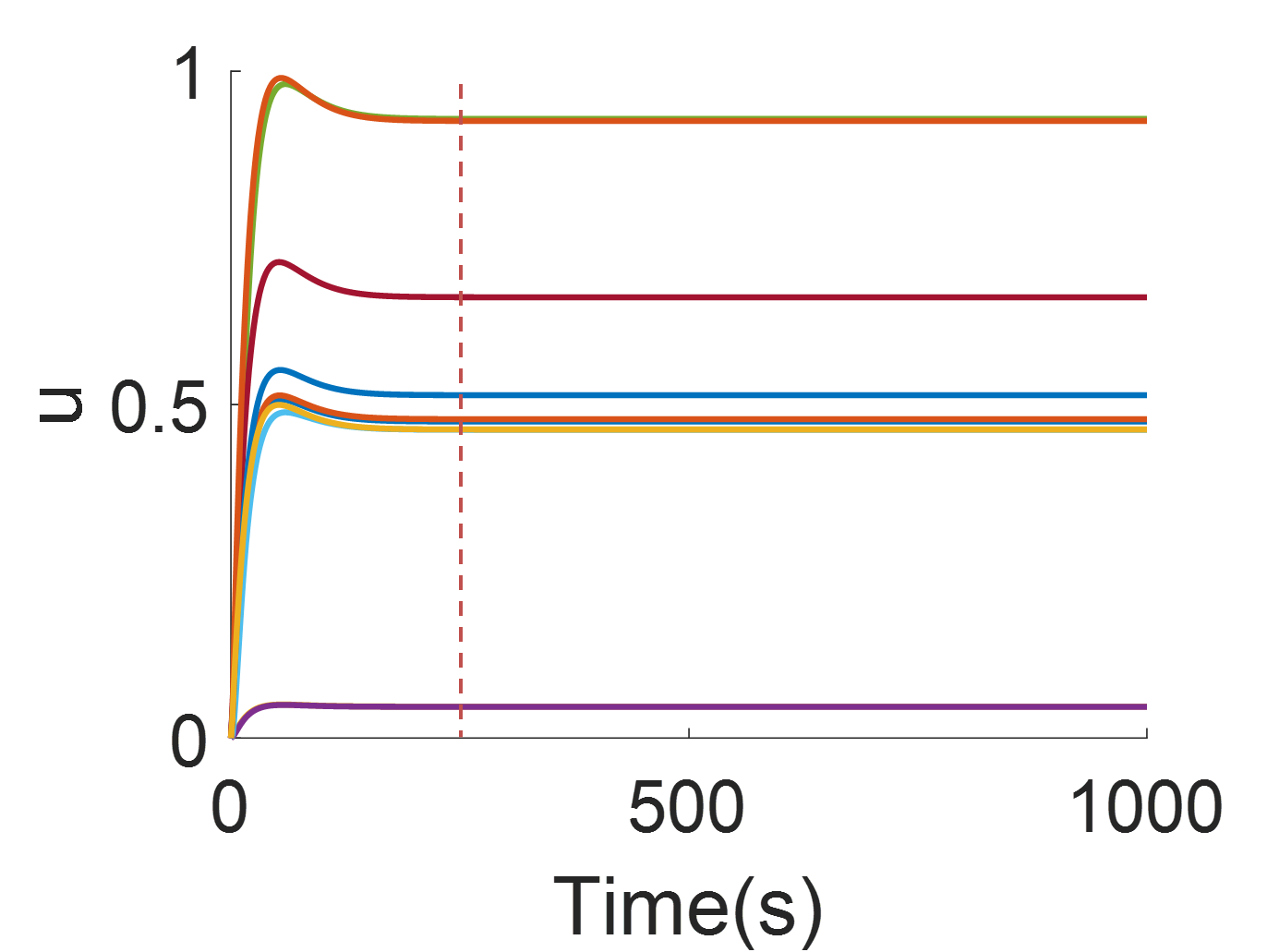}
    \caption{Links 2, 4 fail.}
    \label{fig:2}
  \end{subfigure}
  \begin{subfigure}[b]{0.24\textwidth}
    \includegraphics[width=\textwidth]{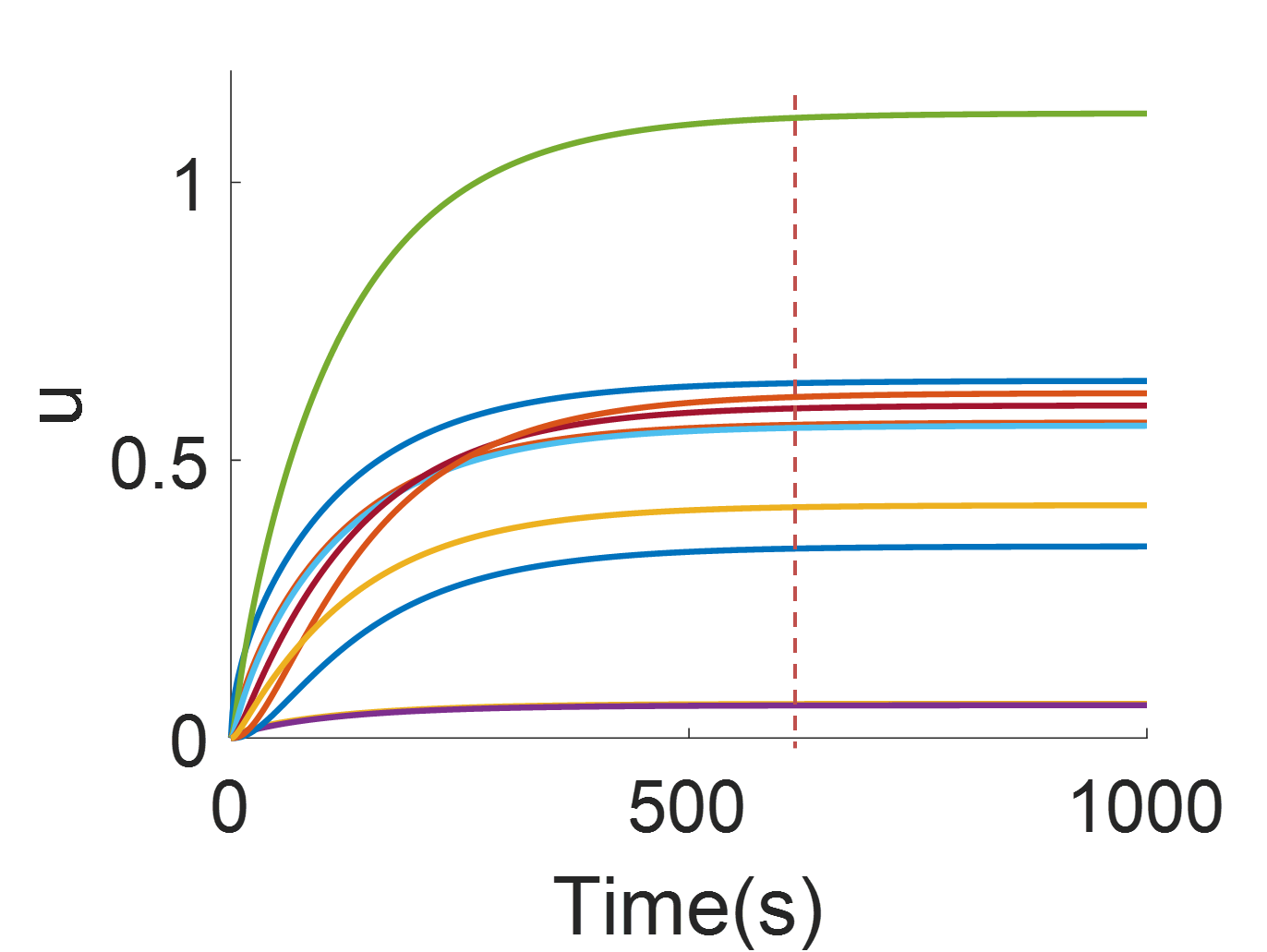}
    \caption{Links 2, 4, 9 fail.}
    \label{fig:3}
  \end{subfigure}
  \begin{subfigure}[b]{0.24\textwidth}
    \includegraphics[width=\textwidth]{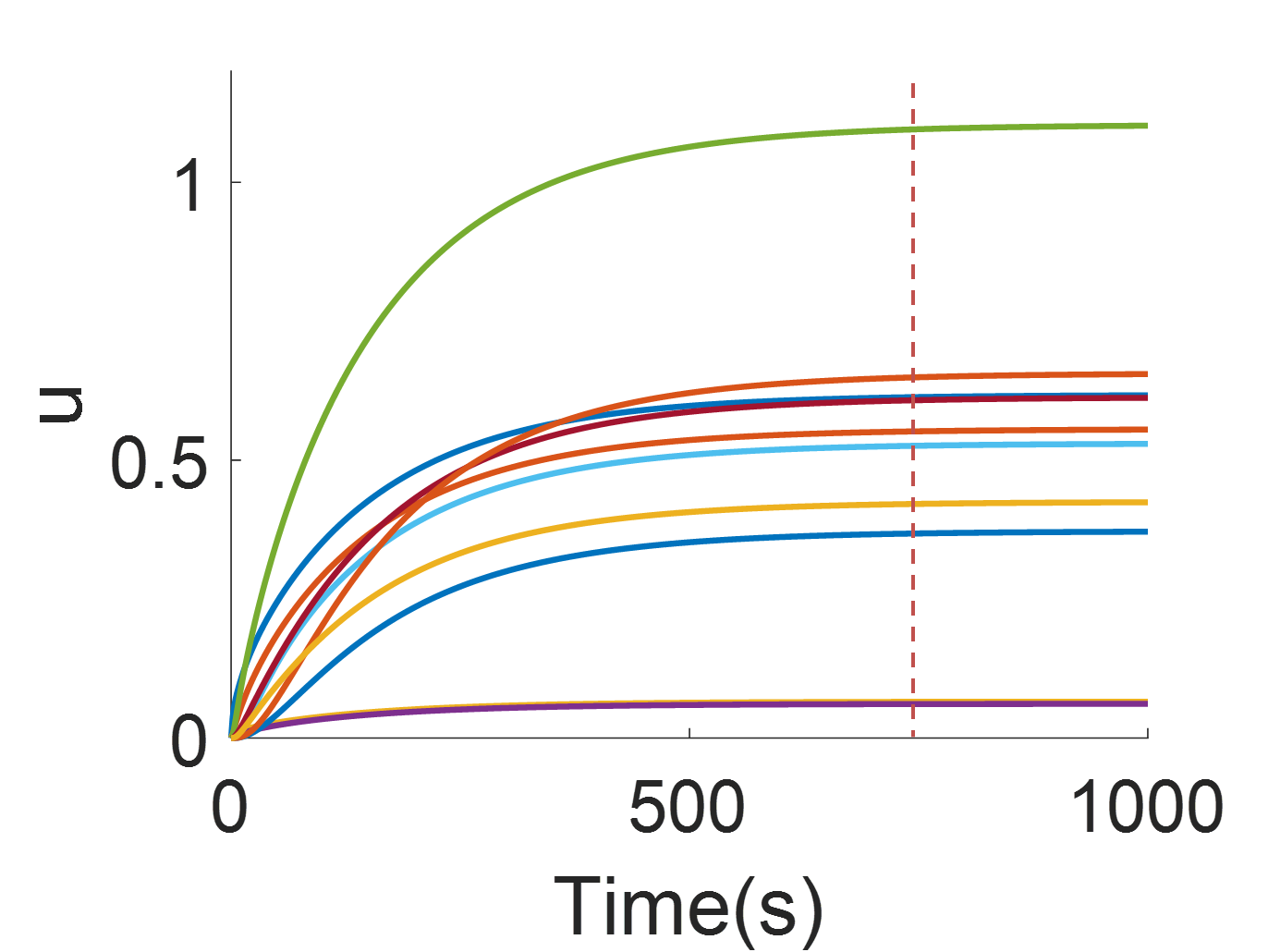}
    \caption{No communication.}
    \label{fig:4}
  \end{subfigure}
  \caption{Convergence time with and without communication. Curves illustrate the adjustable power at all nodes.}
  \label{fig:conv}
\end{figure}

\subsection{Importance of each individual communication link}
We verify that the failures of communication links have more significant impact on the cost, if the corresponding power lines have small susceptances. For $h = 1$, we study the control Eqs. (\ref{eq:newcontrol1}) and (\ref{eq:newcontrol2}) under three communication link failures. In the left figure, nodes adjacent to links $(1,2),(2,5)$ are controlled by Eq. (\ref{eq:newcontrol1}). The corresponding power lines have larger susceptances 0.5 and 1. In the right figure, nodes adjacent to links $(4,5),(7,8)$ are controlled by Eq. (\ref{eq:newcontrol1}). The corresponding power lines have smaller susceptances 0.2 and 0.1. The total costs in the steady states are 26.17 and 34.36, for the left and right figures, respectively. We observe that the cost is higher if communication fails between nodes connected by power lines with smaller susceptances.

\begin{figure}[h]
  \begin{subfigure}[b]{0.22\textwidth}
    \includegraphics[width=\textwidth]{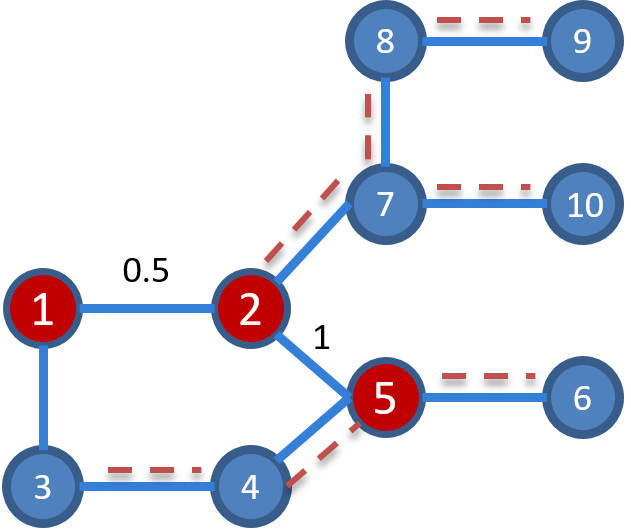}
    \label{fig:21}
  \end{subfigure}
  \begin{subfigure}[b]{0.22\textwidth}
    \includegraphics[width=\textwidth]{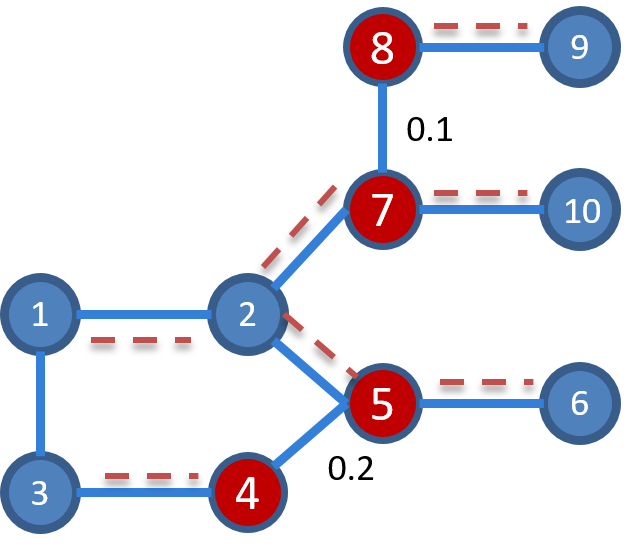}
    \label{fig:22}
  \end{subfigure}
  \caption{Communication link failures.}
\end{figure}

\subsection{Delayed control}
We study the performance of delayed control Eqs. (\ref{eq:control3}) and (\ref{eq:control4}). By fixing $h = 1$, in Fig. \ref{fig:delay}, the adjustable power generation under the control without delay ($T=0$) is illustrated by the left figure, with total cost 39.11. The control with delay $T=30$ seconds is illustrated by the right figure with total cost 27.50. The convergence times are close (differ by 30 seconds), while the cost under the delayed control is $30\%$ lower than the cost under the original control.

\begin{figure}[h]
  \begin{subfigure}[b]{0.24\textwidth}
    \includegraphics[width=\textwidth]{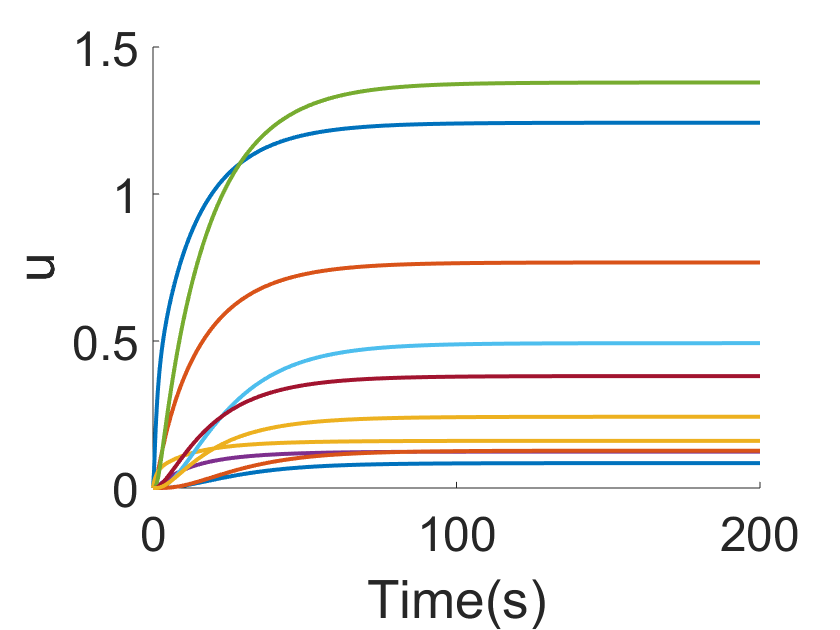}
    \label{fig:31}
  \end{subfigure}
  \begin{subfigure}[b]{0.24\textwidth}
    \includegraphics[width=\textwidth]{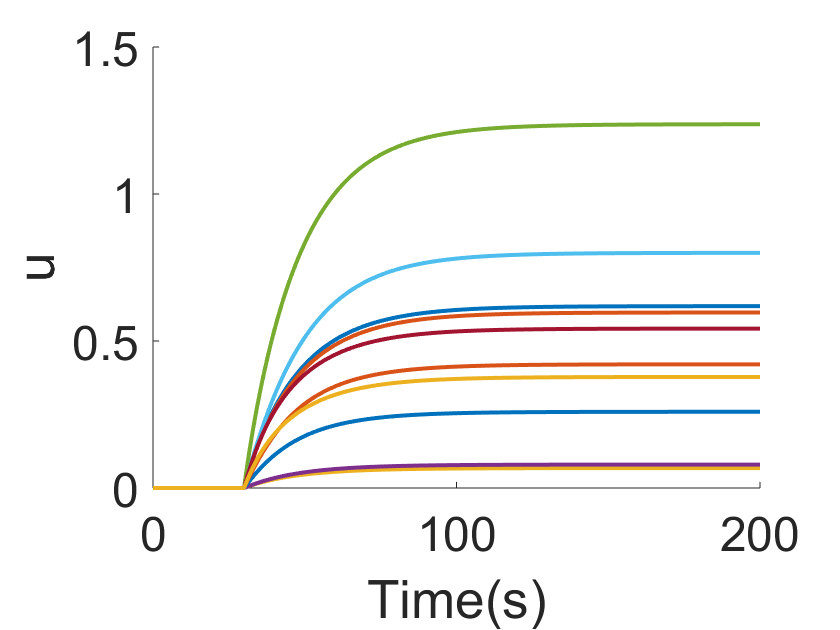}
    \label{fig:32}
  \end{subfigure}

  \caption{Delayed control (left: $T=0$, right: $T=30$ s).}
  \label{fig:delay}
\end{figure}

\subsection{General convex cost function}
We evaluate the performance of the control Eqs. (\ref{eq:virtualcontrol2}) and (\ref{eq:control2}), for a cubic cost function $f_j(u_j) = a_j |u_j|^3 / 3$. The optimal cost is 8.84. The costs obtained by the decentralized integral controller for controller gains $h$ are illustrated in Fig. \ref{fig:gaincubic}. The curve is similar to the curve in Fig. \ref{fig:gain} for a quadratic cost. The results show that the integral control can be applied to arbitrary convex cost function.

\begin{figure}[h]
\centering
\includegraphics[width=.65\linewidth]{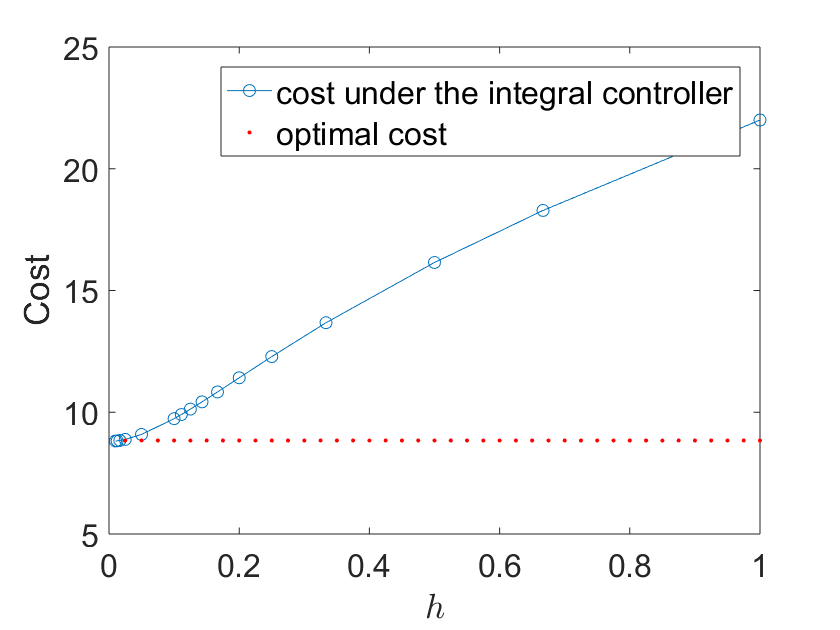}
\caption{Cost for the control under a cubic cost function.}
\label{fig:gaincubic}
\end{figure}

\section{Conclusion}\label{sc:conclusion}
We studied a decentralized integral control for joint frequency regulation and economic dispatch. We derived conditions for the control to achieve near-optimal cost, and observed a tradeoff between the cost and the convergence time. We studied the role of communication in reducing the convergence time. Moreover, we extended the control to handle arbitrary convex costs and power generation capacity constraints. Numerical results show that a delayed control reduces the cost significantly with similar convergence time.

\bibliographystyle{IEEEtran}
\bibliography{power}
\end{document}